\documentclass[12pt]{article}

\usepackage[utf8]{inputenc}
\usepackage[T1]{fontenc}
\usepackage[english]{babel}

\usepackage{lmodern}
\usepackage{microtype}

\usepackage[margin=1in]{geometry}

\usepackage{amsmath, amssymb, amsfonts, amsthm, mathtools, mathrsfs}

\usepackage{graphicx}
\usepackage{tikz}
\usepackage{tikz-cd}

\usepackage{enumitem}

\usepackage{hyperref}
\hypersetup{
    colorlinks=true,
    linkcolor=blue,
    citecolor=blue,
    urlcolor=blue
}

\usepackage{cleveref}
\theoremstyle{plain}
\newtheorem{theorem}{Theorem}[section]
\newtheorem{lemma}[theorem]{Lemma}
\newtheorem{proposition}[theorem]{Proposition}

\theoremstyle{definition}
\newtheorem{definition}[theorem]{Definition}

\theoremstyle{remark}
\newtheorem{remark}[theorem]{Remark}

\title{\bf A Jacobi Coupling Construction   on Associated Bundles}

\author{
Emmanuel Davakan\thanks{IMSP, Porto-Novo, Benin. Email: emmanuel.davakan@imsp-uac.org} \\
Djideme Franck Houenou\thanks{IMSP, Porto-Novo, Benin. Email: rdjeam@imsp-uac.org} \\
A\"issa Wade\thanks{Mathematics Department, Penn State University, University Park, USA. Email: wade@math.psu.edu}
}

\date{}

\begin{document}

\maketitle

\begin{abstract}
We extend the Sternberg--Weinstein coupling construction to the  Jacobi geometry setting.
Starting from a Jacobi  Hamiltonian $G$-space and a principal bundle equipped with a connection whose curvature satisfies some nondegeneracy condition, we show that the associated bundle naturally carries a Jacobi structure compatible with the canonical ones on the fibers.
This construction provides a unified framework encompassing the symplectic, locally conformal symplectic, and contact cases. It  reveals new coupling phenomena related to the presence of the Reeb vector field.
\end{abstract}
	

	\tableofcontents

\section{Introduction}
The construction of symplectic structures on associated bundles originates in the foundational works of Sternberg \cite{Ste77} and Weinstein \cite{Wei80}. This is closely related to the study of symplectic fibrations  that has attracted a lot of interest in the 1980's (see \cite{GLSW83, GS84}).
Let $(F,\omega)$ be a symplectic manifold equipped with a Hamiltonian action of a Lie group $G$, and let $P \to B$ be a principal $G$-bundle endowed with a connection. The associated bundle $M = P \times_G F$ carries a canonical symplectic structure provided that the connection is \emph{fat}, namely,  the curvature, composed with the moment map $\mu : F \to \mathfrak{g}^*$, induces a nondegenerate $2$-form on the horizontal distribution. This construction, known as the \emph{coupling form}, plays a central role in symplectic geometry. In \cite{HWD25}, we provide necessary and sufficient conditions for the existence of a locally conformal presymplectic coupling 2-form on the total space of a locally conformal symplectic fibration.

\medskip
This mechanism admits natural extensions beyond the symplectic setting. In the framework of locally conformally symplectic geometry, it is shown in \cite{HRS15} that coupling-type constructions persist: although the condition $d\omega = 0$ is replaced by $d\omega = \theta \wedge \omega$, global structures can still be obtained from a Hamiltonian action and a connection satisfying appropriate compatibility conditions.
In contact geometry, related phenomena arise in the study of group actions and fibrations, notably in the work of Lerman \cite{Ler04}. Although these constructions do not take the form of a direct coupling theorem analogous to the symplectic case, they highlight the fundamental role played by curvature in the construction of global geometric structures.

\medskip 
These developments reflect a common geometric principle, naturally formulated in Jacobi geometry setting. A Jacobi structure on a smooth manifold $N$ is a pair $(\Lambda_N,E_N)$ consisting of a bivector field $\Lambda_N$ and a vector field $E_N$ satisfying
\[
[\Lambda_N,\Lambda_N]_{\mathrm{SN}} = 2E_N \wedge \Lambda_N, 
\qquad 
[\Lambda_N,E_N]_{\mathrm{SN}} = 0,
\]
where $[\cdot,\cdot]_{\mathrm{SN}}$ denotes the Schouten--Nijenhuis bracket. Introduced independently by Kirillov \cite{Kir76} and Lichnerowicz \cite{Lic78}, these structures encompass Poisson structures (when $E_N=0$), as well as contact and locally conformally symplectic structures (see  for instance \cite{DLM91,Vai85}).

\medskip 
The goal of the present paper is to establish a coupling-type theorem for Jacobi structures, thereby unifying the constructions described above. We show that a Hamiltonian Jacobi action, together with a principal bundle and a connection satisfying a nondegeneracy condition gives rise to a Jacobi structure on the associated bundle. We will prove our main result:

\medskip

\noindent\textbf{Theorem.}
Let $(F,\Lambda_F,E_F)$ be a Jacobi manifold, and let $G$ be a Lie group acting on $F$ by Jacobi automorphisms. Assume that this action is Hamiltonian, with a $G$-equivariant moment map $\mu : F \to \mathfrak{g}^*$. Let $P \to B$ be a principal $G$-bundle equipped with a connection $\omega$ with curvature $\mathrm{Curv}_\omega$. If $\omega$ is $\mu(F)$-fat, then the associated bundle $M = P \times_G F$ carries a Jacobi structure $(\Lambda,E)$ whose restriction to each fiber is isomorphic to $(\Lambda_F,E_F)$.

\medskip

Our construction provides a concrete realization, in the setting of associated bundles, of the coupling approach to Jacobi structures introduced by Vaisman \cite{Vai2004}. The principal connection induces a splitting of $TM$ into horizontal and vertical components, while the Jacobi structure on $F$ determines a fiberwise Jacobi structure on $M$. The interaction between these components is governed by the curvature via the infinitesimal action of $\mathfrak{g}$. In particular, the appearance of the term $\Omega_\omega\,E_F$ reflects an intrinsic coupling between the curvature and the Reeb vector field, a phenomenon that is specific to the Jacobi setting.
More generally, our construction extends symplectic coupling forms, their locally conformally symplectic counterparts, and related constructions in contact geometry. The curvature compatibility conditions involve an additional term of the form $\Omega_\omega\,E_F$, which has no analogue in symplectic or Poisson geometry, where the structure is entirely governed by bivector fields.

\medskip
Finally, our result may be viewed as a Jacobi analogue of the construction of the third author in \cite{AGAG08}, where coupling Dirac structures are obtained on associated bundles arising from Hamiltonian Poisson spaces.  Integration of Poisson fibrations coming from coupling Poisson structures was studied in \cite{BrahicFernandes2008}. Our construction  also produces a natural class of Jacobi fibrations, in the sense that the associated bundle inherits a Jacobi structure compatible with the fibration. In this perspective, our construction connects with the theory of conformal Jacobi fibrations studied in \cite{Wade2013}, where such fibrations are investigated through their integration to contact groupoids.

	\section{Preliminaries}
\label{sec:prelim}

\subsection{A brief review: associated bundles}

\noindent We briefly recall  classical definitions related to principal bundles (see \cite{KN63}).
\begin{definition}
A \emph{principal $G$-bundle} consists of a smooth manifold $P$ together with a smooth, free and proper right action of a Lie group $G$ such that the quotient $B := P/G$ is a smooth manifold and the projection 
$\pi : P \to B$ is a smooth surjective submersion. Moreover, $\pi$ is  a locally trivial fiber bundle in the sense that for every $b \in B$ there exists an open neighborhood $U \subseteq B$  of $b$ and a $G$-equivariant diffeomorphism
\[
\Phi : \pi^{-1}(U) \to U \times G
\]
satisfying
\[
\Phi(p \cdot g) = \Phi(p)\cdot g,
\]
where $G$ acts on $U \times G$ by right multiplication on the second factor.
\end{definition}

\begin{definition}
Let $\pi : P \to B$ be a principal $G$-bundle and $F$ a smooth $G$-manifold. 
The associated bundle $P \times_G F$ is the quotient of $P \times F$ by the right action
\[
(p,f)\cdot g := (pg,\, g^{-1}\!\cdot f).
\]
Denoting by $[p,f]$ the class of $(p,f)$, the projection $\pi_\rho : P \times_G F \to B$ given by : 
\[
\pi_\rho([p,f]) := \pi(p)
\]
is well-defined and makes $P \times_G F$ into a smooth fiber bundle over $B$ with typical fiber $F$.
\end{definition}

\subsection{Jacobi manifolds}

\noindent Jacobi structures were introduced independently by Kirillov \cite{Kir76} and Lichnerowicz \cite{Lic78}. 
We briefly recall the definition of a Jacobi structure in the sense of Lichnerowicz and refer the reader to \cite{Vai85} for a detailed exposition.

\begin{definition}
A \emph{Jacobi structure} on a smooth manifold $F$ (in Lichnerowicz's sense) is given by a pair $(\Lambda_F,E_F)$, 
where $\Lambda_F \in \mathfrak{X}^2(F)$ is a bivector field and 
$E_F \in \mathfrak{X}(F)$ is a vector field such that: 
\[
[\Lambda_F,\Lambda_F]=2E_F\wedge\Lambda_F, 
\qquad {\rm and} \qquad 
[E_F,\Lambda_F]=0,
\]
where $[\cdot,\cdot]$ denotes the Schouten bracket.
\end{definition}

\noindent The associated Jacobi bracket on $C^\infty(F)$ is defined by:
\[
\{u,v\}
=
\Lambda_F(du,dv)+u E_F(v)-v E_F(u).
\]

\begin{remark}
\noindent This notion encompasses several classical geometries:
\begin{itemize}
\item Poisson manifolds correspond to the case $E_F=0$;
\item any contact manifold $(F,\alpha)$ defines a Jacobi structure whose vector field $E_F$ is the Reeb field and whose bivector field $\Lambda_F$ is induced by $d\alpha$ on $\ker \alpha$;
\item any locally conformal symplectic manifold $(F,\omega,\theta)$ satisfying $d\omega=\theta \wedge \omega$ defines a Jacobi structure by setting $\Lambda_F=\omega^{-1}$ and $E_F=\theta^\sharp$.
\end{itemize}
\end{remark}

\subsection{Hamiltonian Actions in Jacobi Geometry}
\noindent Let $(F,\Lambda_F,E_F)$ be a Jacobi manifold.
A vector field $X \in \mathfrak{X}(F)$ is said to be \emph{Hamiltonian}
if there exists a function $h \in C^\infty(F)$ such that $X = X_h,$ where
\[
X_h := \Lambda_F^\sharp(dh) + hE_F,
\]
and $\Lambda_F^\sharp : T^*F \to TF$ denotes the bundle map defined by
$\beta(\Lambda_F^\sharp(\alpha)) = \Lambda_F(\alpha,\beta)$ for all $\alpha,\beta \in T^*F$.

\begin{definition}\label{Hamiltonian Jacobi action}
Let $(F,\Lambda_F,E_F)$ be a Jacobi manifold, and let $G$ be a Lie group with Lie algebra $\mathfrak{g}$.  
A \emph{Hamiltonian Jacobi action} of $G$ on $F$ is a smooth action of $G$ on $F$ by Jacobi automorphisms for which there exists a smooth moment map
$\mu : F \longrightarrow \mathfrak{g}^*,$ such that:

\begin{enumerate}
\item For every $X \in \mathfrak{g}$, the fundamental vector field $X_F$ associated with $X$ coincides with the Hamiltonian vector field of the function $\mu^X : F \to \mathbb{R}$ defined by
\[
\mu^X(f) := \langle \mu(f), X \rangle, \qquad \forall f \in F.
\]
i.e $X_F = X_{\mu^X}.$

\item The moment map $\mu$ is $G$-equivariant with respect to the given action on $F$ and the coadjoint action on $\mathfrak{g}^*$, i.e.
\[
\mu(g \cdot f) = \operatorname{Ad}^*_{g^{-1}}(\mu(f)),
\qquad \forall g \in G,\; f \in F.
\]
\end{enumerate}

\noindent In this case, $(F,\Lambda_F,E_F)$ is called a \emph{Jacobi Hamiltonian $G$-space}.
\end{definition}

\begin{remark}
\noindent In the setting of locally conformally symplectic manifolds, 
the twisted Hamiltonian actions introduced in \cite{HRS15} 
naturally correspond to the Hamiltonian Jacobi actions defined as above.
\end{remark}

\subsection{Connections and Curvature}

\noindent Let $\pi: P \to B$ be a principal $G$-bundle. A connection on $P$ is a 
$1$-form $\omega \in \Omega^1(P,\mathfrak{g})$ satisfying:
\begin{enumerate}
\item $\omega_p(X^*_p)=X$ for all $p\in P$, $X \in \mathfrak{g}$,  and where $X^*$ is the fundamental vector field associated with $X$;
\item $R_g^*\omega=\operatorname{Ad}_{g^{-1}} \circ \omega$ for every $g \in G$.
\end{enumerate}
\noindent Such a connection induces the decomposition
\[
T_pP = H_p \oplus V_p, 
\qquad 
H_p=\ker\omega_p,
\qquad 
V_p=\ker(d\pi_p).
\]
The curvature of $\omega$ is the $\mathfrak{g}$-valued $2$-form
\[
\Omega_\omega = d\omega + \tfrac{1}{2}[\omega,\omega],
\]
which satisfies
\[
R_g^*\Omega_\omega=\operatorname{Ad}_{g^{-1}}\circ\Omega_\omega,
\qquad 
d\Omega_\omega+[\omega,\Omega_\omega]=0.
\]

\begin{definition}\label{fat connection}
\noindent Let $S \subset \mathfrak{g}^*$. The connection $\omega \in \Omega^2(P,\mathfrak{g})$ is said to be 
$S$-\emph{fat} if, for every $s \in S$, the form
\[
K_s(X,Y)=\langle s,\Omega_\omega(X,Y)\rangle,
\quad X,Y\in H_p,
\]
is nondegenerate.
\end{definition}

\subsection{Induced Connection on the Associated Bundle}

\noindent
Let $M = P \times_G F$ be the associated bundle corresponding to a smooth left action of $G$ on $F$, where $G$ acts on $P \times F$ by
\[
(p,f)\cdot g = (pg,\, g^{-1}\cdot f).
\]
Let $\omega$ be a $1$-from connection on $P$, and denote by $H_p = \ker \omega_p$ the horizontal subspace at $p \in P$.
Then $\omega$ induces a connection $\Gamma$ on $M$ whose horizontal distribution $\mathcal{H}$  defined by
\[
\mathcal{H}_{[p,f]}
=
(d\pi_{\mathrm{quot}})_{(p,f)}\bigl(H_p \times \{0\}\bigr),
\]
where $\pi_{\mathrm{quot}} : P \times F \to P \times_G F$ denotes the quotient map.

\noindent This definition is independent of the choice of representative $(p,f)$ of $[p,f]$. 
Indeed, the $G$-equivariance of the connection implies
$dR_g(H_p)=H_{pg},$ and this is compatible with the quotient construction of $M$.
Thus, $\mathcal{H}$ defines a smooth horizontal distribution on $M$ which is complementary to the vertical subbundle $\mathcal{V}:=\ker(d\pi_M)$, where $\pi_M : M \to B$ is the natural bundle projection.
The curvature of $\Gamma$ is given by its failure to be integrable and takes values in the vertical bundle. More precisely, let $X,Y \in \mathfrak{X}(B)$ and let $X^{\rm hor},Y^{\rm hor}$ denote their horizontal lifts to $P$. Then, for every $[p,f] \in M$,
\[
R_{\mathcal{H}}(X,Y)([p,f])
=
(d\pi_{\mathrm{quot}})_{(p,f)}
\Bigl(
0,\,
\rho^F\bigl(\Omega_\omega(X^{\rm hor},Y^{\rm hor})(p)\bigr)\big|_{f}
\Bigr),
\]
where $\rho^F : \mathfrak{g} \to \mathfrak{X}(F)$ denotes the infinitesimal action of $\mathfrak{g}$ on $F$.
\section{Main result} 

\subsection{Statement of the main result} 

\noindent In this Section, our main  goal is to prove the following theorem:

\begin{theorem}\label{MainTheorem}
Let $G$ be a Lie group with Lie algebra $\mathfrak g$, and  $P \to B$, a principal $G$-bundle. Let $(F,\Lambda_F,E_F)$ be a Jacobi Hamiltonian $G$-space.
Let $\omega \in \Omega^1(P,\mathfrak g)$ be a connection on $P$, which is $\mu(F)$-fat in the sense of Definition~\ref{fat connection}. Then the associated bundle
$M = P \times_G F$ carries a natural Jacobi structure $(\Lambda,E)$ such that, for every $b \in B$, the restriction of $(\Lambda,E)$ to the fiber $M_b \simeq F$ is isomorphic, as a Jacobi structure, to $(\Lambda_F,E_F)$.
\end{theorem}

\medskip

\subsection{Proof of the main result}

\noindent The proof of this result will be divided into two parts. In the first part, we construct a pair $(\widetilde \Lambda, \widetilde E)$ consisting of a bivector field $\widetilde \Lambda$ and a vector field $\widetilde E$ on $P \times F$, which we then descend to the quotient $M = P \times_G F$ to obtain a pair $(\Lambda,E)$. The second part is devoted to showing that $(\Lambda,E)$ defines a Jacobi structure on $M = P \times_G F$.

\subsubsection{Tensors on $P \times_G F$}

\noindent
\textbf{1. Construction on $P \times F$.}

\noindent The connection on $P$ induces the decomposition
\[
T_{(p,f)}(P\times F)=H_p\oplus V_p\oplus T_fF.
\]
We extend $(\Lambda_F,E_F)$ by setting
\[
\Lambda^{\mathrm{ver}}_{(p,f)}=\Lambda_F(f),
\qquad
\widetilde{E}_{(p,f)}=E_F(f).
\]
We define a 2-form on $H_p$ by
\begin{equation}
\label{horizontal 2-form}
\Omega_{(p,f)}(X,Y)
=
\langle\mu(f),\Omega_\omega(X,Y)\rangle=(K_{\mu(f)} (X,Y))(p).
\end{equation}

\noindent The assumption that $\omega$ is $\mu(F)$-fat implies that $\Omega$ is nondegenerate; hence its inverse defines a horizontal bivector field $\Lambda^{\mathrm{hor}}$. We then set
\[
\widetilde{\Lambda}=\Lambda^{\mathrm{ver}}+\Lambda^{\mathrm{hor}}.
\]

\noindent \textbf{2. Invariance and descent to $P \times_G F$.}

\noindent We consider the diagonal action of $G$ on $P\times F$ defined by
\[
\rho : G\times (P\times F) \to P\times F,
\qquad
\rho(g,(p,f))=(pg^{-1},\,\rho^F(g,f)).
\]
where  $\rho^F: G \times F \to F$ denotes also  $G$-action on $F$. 

\medskip

\noindent $\bullet$ \textbf{Invariance.}
Since the action $\rho^F$ preserves the Jacobi structure on $F$, we have
\[
(\rho^F_g)_*\Lambda_F = \Lambda_F,
\qquad
(\rho^F_g)_*E_F = E_F,
\]
hence $(\Lambda^{\mathrm{ver}},\widetilde E)$ is invariant.
By equivariance of the curvature,
\[
R_g^*\Omega_\omega = \operatorname{Ad}_{g^{-1}}\circ \Omega_\omega,
\]
and the moment map satisfies
\[
\mu(g\cdot f)=\operatorname{Ad}^*_{g^{-1}}\mu(f).
\]
It follows that, for all $g \in G$ and horizontal vectors $X,Y$,
\[
(\rho_g^*\Omega)(X,Y)
= \big\langle \operatorname{Ad}^*_{g^{-1}}\mu(f),\, \operatorname{Ad}_{g^{-1}}\Omega_\omega(X,Y)\big\rangle
= \langle \mu(f), \Omega_\omega(X,Y)\rangle
= \Omega(X,Y),
\]
using the duality between the adjoint and coadjoint representations. Hence $\Omega$ is $G$-invariant. Since $\Lambda^{\mathrm{hor}}$ is defined as the inverse of $\Omega$ on the horizontal distribution, which is itself preserved by the action, it follows that $\Lambda^{\mathrm{hor}}$ is also $G$-invariant. Therefore $(\widetilde{\Lambda},\widetilde E)$ is $G$-invariant.

\medskip

\noindent $\bullet$ \textbf{Descent to the quotient.}
Consider the canonical projection
\[
\begin{tikzcd}
P\times F \arrow[r, "\rho_g"] \arrow[d, "\pi"']
& P\times F \arrow[d, "\pi"] \\
M \arrow[r, equals]
& M
\end{tikzcd}
\]
\medskip

\noindent \emph{Projectability:}
Let $(p,f)$ and $(pg^{-1},g\cdot f)$ be two representatives of the same equivalence class.
Then, for every $v\in T_{(p,f)}(P\times F)$,
\[
(d\pi)_{(pg^{-1},g\cdot f)} \circ (d\rho_g)_{(p,f)}(v)
=
(d\pi)_{(p,f)}(v).
\]
Since $(\rho_g)_*\widetilde{\Lambda}=\widetilde{\Lambda}$, it follows that the value of $\widetilde{\Lambda}$ is independent of the representative of $[p,f]$.
Thus $\widetilde{\Lambda}$ is projectable and there exists a unique bivector field $\Lambda$ on $M$ such that
\[
(d\pi)_*\widetilde{\Lambda} = \Lambda.
\]
Similarly, $\widetilde E$ descends to a vector field $E$ on $M$.

\medskip

\noindent It follows that $(\widetilde{\Lambda},\widetilde E)$ is $G$-invariant and projectable, hence  it gives a well-defined pair $(\Lambda,E)$ on
\[
M = P\times_G F.
\]
Recall that the connection on $P$ induces a connection $\Gamma$ on $M$ by transporting the horizontal distribution. We thus obtain a decomposition
\[
TM = \mathcal{H} \oplus \mathcal{V},
\]
where $\mathcal{H} $ and $\mathcal{V} $ denote respectively the horizontal and vertical subbundles.

\noindent For $X,Y \in \mathfrak{X}(B)$, let $X^{\mathrm{hor}}, Y^{\mathrm{hor}}$ denote their horizontal lifts to $M$. The curvature of the connection is defined by
\[
\operatorname{Curv}_\Gamma(X,Y)
:=
\operatorname{proj}_{\mathrm{ver}}
\bigl([X^{\mathrm{hor}},Y^{\mathrm{hor}}]\bigr).
\]

\medskip

\begin{remark}\label{rmq3.1}
Let $\Lambda = \Lambda^{\mathrm{ver}}+\Lambda^{\mathrm{hor}}$ and let $E = E^{\mathrm{ver}}$
be the tensors on $M$ defined above relative to the decomposition
\[
TM = \mathcal{H}  \oplus \mathcal{V} .
\]
We have
\[
[\Lambda,\Lambda]
=
[\Lambda^{\mathrm{hor}},\Lambda^{\mathrm{hor}}]
+2
[\Lambda^{\mathrm{hor}},\Lambda^{\mathrm{ver}}]
+
[\Lambda^{\mathrm{ver}},\Lambda^{\mathrm{ver}}],
\]
and
\[
E \wedge \Lambda
=
 E^{\mathrm{ver}}\wedge \Lambda^{\mathrm{ver}}+ E^{\mathrm{ver}}\wedge \Lambda^{\mathrm{hor}}.
\]
The vertical part of the Jacobi identity $[\Lambda,\Lambda] = 2 E \wedge \Lambda$  is:
\[
[\Lambda^{\mathrm{ver}},\Lambda^{\mathrm{ver}}]=2 E^{\mathrm{ver}}\wedge \Lambda^{\mathrm{ver}}.
\]
In other words, the Jacobi identity is equivalent to the system
\[
\begin{cases}
[\Lambda^{\mathrm{ver}},\Lambda^{\mathrm{ver}}]
=
2 E^{\mathrm{ver}}\wedge \Lambda^{\mathrm{ver}}, \\[0.4em]
[\Lambda^{\mathrm{hor}},\Lambda^{\mathrm{hor}}]
+
2[\Lambda^{\mathrm{ver}},\Lambda^{\mathrm{hor}}]
=
2 E^{\mathrm{ver}}\wedge \Lambda^{\mathrm{hor}}.
\end{cases}
\]
\end{remark}

\subsubsection{Jacobi identity}

\noindent In view of Remark~\ref{rmq3.1}, to show that $(\Lambda,E)$ defines a Jacobi structure on $M$, it is enough to verify the following three points:

\begin{itemize}
\item \textbf{Vertical part:} $(\Lambda^{\mathrm{ver}},E^{\mathrm{ver}})$ defines a vertical Jacobi structure. This follows directly from the fact that $(\Lambda_F,E_F)$ is a Jacobi structure on the fiber $F$.

\item \textbf{Horizontal part:} 
\[
[\Lambda^{\mathrm{hor}},\Lambda^{\mathrm{hor}}]^{\mathrm{hor}} = 0,
\]
that is, $\Lambda^{\mathrm{hor}}$ defines a horizontal Poisson structure. This follows from the fact that the 2-form $\Omega$ defined as in Equation (\ref{horizontal 2-form}) is closed in horizontal directions.

\item \textbf{Mixed part:}
\[
\frac{1}{2} [\Lambda^{\mathrm{hor}},\Lambda^{\mathrm{hor}}]^{\mathrm{mix}}
+
[\Lambda^{\mathrm{ver}},\Lambda^{\mathrm{hor}}]
=
 E^{\mathrm{ver}}\wedge \Lambda^{\mathrm{hor}}.
\]
This follows from the Hamiltonian condition
\[
X_F=\Lambda_F^\sharp(d\mu^X)+\mu^X E_F,
\]
which encodes the compatibility between the action and the Jacobi structure.
\end{itemize}

\begin{lemma}\label{Lemme 3.2}
Let $M$ be a manifold endowed with a decomposition
\[
\Lambda=\Lambda^{\mathrm{ver}}+\Lambda^{\mathrm{hor}},
\qquad E= E^{\mathrm{ver}},
\]
where $\Lambda^{\mathrm{ver}}$ is vertical, $\Lambda^{\mathrm{hor}}$ is horizontal, and $E$ is vertical.
Let $\alpha$ be a horizontal 1-form, and let $\beta,\gamma$ be two vertical 1-forms. Then
\[
\left( \frac12[\Lambda,\Lambda]-E\wedge\Lambda \right) (\alpha,\beta,\gamma) = 0
\quad \Longleftrightarrow \quad
[\Lambda^{\mathrm{hor}},\Lambda^{\mathrm{ver}}](\alpha,\beta,\gamma)=0
\quad \Longleftrightarrow \quad
\mathcal{L}_{X^{\mathrm{hor}}}\Lambda^{\mathrm{ver}}(\beta,\gamma)=0,
\] where $X^{\mathrm{hor}}=(\Lambda^{\mathrm{hor}})^\sharp(\alpha).$
\end{lemma}

\begin{proof}
\noindent We evaluate the Jacobi identity
\[
\frac{1}{2}[\Lambda,\Lambda] = E \wedge \Lambda
\]
on $(\alpha,\beta,\gamma)$, where $\alpha$ is horizontal and $\beta,\gamma$ are vertical.
By Remark~\ref{rmq3.1}, only the mixed term contributes, since
\[
[\Lambda^{\mathrm{ver}},\Lambda^{\mathrm{ver}}](\alpha,\beta,\gamma)=0,
\qquad
[\Lambda^{\mathrm{hor}},\Lambda^{\mathrm{hor}}](\alpha,\beta,\gamma)=0,
\]
and $(E\wedge\Lambda)(\alpha,\beta,\gamma)=0$ (as $E$ is vertical and $\alpha$ is horizontal).
Hence the Jacobi identity is equivalent to
\[
[\Lambda^{\mathrm{hor}},\Lambda^{\mathrm{ver}}](\alpha,\beta,\gamma)=0.
\]
Let
\[
X^{\mathrm{hor}}=(\Lambda^{\mathrm{hor}})^\sharp(\alpha),
\qquad
Z=(\Lambda^{\mathrm{ver}})^\sharp(\beta).
\]
Using the definition of the Schouten bracket, we obtain
\[
[\Lambda^{\mathrm{hor}},\Lambda^{\mathrm{ver}}](\alpha,\beta,\gamma)
=
\gamma\big(
[X^{\mathrm{hor}},Z]
-
(\Lambda^{\mathrm{ver}})^\sharp([\alpha,\beta]_\Lambda)
\big), 
\] where \[
[\alpha,\beta]_\Lambda
=
\mathcal{L}_{X^{\mathrm{hor}}}\beta
-
\mathcal{L}_{Z}\alpha
-
d(\Lambda(\alpha,\beta)).
\]
Since $\alpha$ is horizontal and $\beta$ is vertical, we have
\[
\Lambda(\alpha,\beta)=0,
\qquad
\mathcal{L}_Z\alpha=0,
\]
and therefore
\[
[\alpha,\beta]_\Lambda=\mathcal{L}_{X^{\mathrm{hor}}}\beta.
\]
It follows that
\[
[\Lambda^{\mathrm{hor}},\Lambda^{\mathrm{ver}}](\alpha,\beta,\gamma)
=
\gamma\big(
[X^{\mathrm{hor}},Z]
-
(\Lambda^{\mathrm{ver}})^\sharp(\mathcal{L}_{X^{\mathrm{hor}}}\beta)
\big).
\]
By the defining identity for the Lie derivative of a bivector field,
\[
(\mathcal{L}_{X^{\mathrm{hor}}}\Lambda^{\mathrm{ver}})^\sharp(\beta)
=
[X^{\mathrm{hor}}, (\Lambda^{\mathrm{ver}})^\sharp(\beta)]
-
(\Lambda^{\mathrm{ver}})^\sharp(\mathcal{L}_{X^{\mathrm{hor}}}\beta),
\]
we deduce
\[
[\Lambda^{\mathrm{hor}},\Lambda^{\mathrm{ver}}](\alpha,\beta,\gamma)
=
\gamma\!\left((\mathcal{L}_{X^{\mathrm{hor}}}\Lambda^{\mathrm{ver}})^\sharp(\beta)\right).
\]
Therefore,
\[
[\Lambda^{\mathrm{hor}},\Lambda^{\mathrm{ver}}](\alpha,\beta,\gamma)=0
\quad \Longleftrightarrow \quad
\gamma\!\left((\mathcal{L}_{X^{\mathrm{hor}}}\Lambda^{\mathrm{ver}})^\sharp(\beta) \right)=0,~ ~  \forall ~ \beta, \gamma ~{\rm vertical}.
\]
Since $\beta$ and $\gamma$ are arbitrary among vertical 1-forms, this is equivalent to
$\mathcal{L}_{X^{\mathrm{hor}}}\Lambda^{\mathrm{ver}}=0.$

\end{proof}


\begin{lemma}\label{lem:curv-jacobi}
Let $(\Lambda,E)$ be the tensors on $M$ defined as above, with
\[
\Lambda = \Lambda^{\mathrm{ver}} + \Lambda^{\mathrm{hor}},
\qquad
E = E^{\mathrm{ver}},
\]
where $\Lambda^{\mathrm{hor}} \in \Gamma(\wedge^2 H)$ is a horizontal bivector field. 
Let $X,Y \in \mathfrak{X}(B)$. Let $X^{\mathrm{hor}}, Y^{\mathrm{hor}} \in \mathfrak{X}(M)$ denote their horizontal lifts. Let $\alpha,\beta \in \Omega^1(M)$ be horizontal 1-forms such that:
\[
(\Lambda^{\mathrm{hor}})^\sharp(\alpha)=X^{\mathrm{hor}},
\qquad
(\Lambda^{\mathrm{hor}})^\sharp(\beta)=Y^{\mathrm{hor}},
\]
and let $\gamma$ be a vertical 1-form.
Then,
\[ \left( [\Lambda,\Lambda] -2 E \wedge \Lambda \right)(\alpha,\beta,\gamma)=0 \iff
\operatorname{Curv}_\Gamma(X,Y)
=
(\Lambda^{\mathrm{ver}})^\sharp\!\big(d\bigl(\Lambda^{\mathrm{hor}}(\alpha,\beta)\bigr)\big)
+
\Lambda^{\mathrm{hor}}(\alpha,\beta)\,E.
\]
\end{lemma}


\begin{proof}
\noindent We evaluate the Jacobi identity
\[
\frac{1}{2}[\Lambda,\Lambda]=E\wedge\Lambda
\]
on $(\alpha,\beta,\gamma)$, where $\alpha,\beta$ are horizontal $1$-forms and $\gamma$ is vertical.
By Remark~\ref{rmq3.1}, only the terms
\[
\frac{1}{2}[\Lambda^{\mathrm{hor}},\Lambda^{\mathrm{hor}}]
\quad \text{and} \quad
[\Lambda^{\mathrm{ver}},\Lambda^{\mathrm{hor}}]
\]
contribute. Indeed,
\[
[\Lambda^{\mathrm{ver}},\Lambda^{\mathrm{ver}}](\alpha,\beta,\gamma)=0,
\qquad
(E\wedge\Lambda^{\mathrm{ver}})(\alpha,\beta,\gamma)=0,
\]
since $\alpha,\beta$ are horizontal and $\Lambda^{\mathrm{ver}}$ acts only on vertical forms.

\medskip

\noindent $\bullet$ {\bf Horizontal term:}

\noindent We recall the general identity for the Schouten bracket of a bivector field  $\Pi$:
\begin{equation}\label{eq:schouten-eval}
\frac{1}{2}[\Pi,\Pi](\alpha,\beta,\gamma)
=
\sum_{\mathrm{cycl}}
\alpha\!\big([\Pi^\sharp(\beta),\Pi^\sharp(\gamma)]\big)
-
\sum_{\mathrm{cycl}}
\Pi\!\big([\beta,\gamma]_\Pi,\alpha\big),
\end{equation}
\noindent We apply this to $\Pi=\Lambda^{\mathrm{hor}}$ taking into account 
$(\Lambda^{\mathrm{hor}})^\sharp(\alpha), (\Lambda^{\mathrm{hor}})^\sharp(\beta) \in H,$ and $(\Lambda^{\mathrm{hor}})^\sharp(\gamma)=0$.

\medskip

\noindent \emph{Second term:} For any $1$-form $\eta$, we have $ \Lambda^{\mathrm{hor}}(\eta,\gamma)=0$
because $\gamma|_H=0$. Hence
\[
\Lambda^{\mathrm{hor}}([\alpha,\beta]_{\Lambda^{\mathrm{hor}}},\gamma)=0.
\]
Similarly, since $(\Lambda^{\mathrm{hor}})^\sharp(\gamma)=0$, the Koszul bracket gives
$[\beta,\gamma]_{\Lambda^{\mathrm{hor}}}=0$ and $[\gamma,\alpha]_{\Lambda^{\mathrm{hor}}}=0$,
so all cyclic contributions vanish, and therefore
\[
\sum_{\mathrm{cycl}}\Lambda^{\mathrm{hor}}([\alpha,\beta]_{\Lambda^{\mathrm{hor}}},\gamma)=0.
\]

\medskip

\noindent \emph{First term:} Expanding the cyclic sum,
\[
\sum_{\mathrm{cycl}}
\gamma\big([(\Lambda^{\mathrm{hor}})^\sharp(\alpha),(\Lambda^{\mathrm{hor}})^\sharp(\beta)]\big)
=
\gamma([X^{\mathrm{hor}},Y^{\mathrm{hor}}]),
\]
since the two other cyclic terms involve $(\Lambda^{\mathrm{hor}})^\sharp(\gamma)=0$ and vanish.
Thus,
\[
\frac{1}{2}[\Lambda^{\mathrm{hor}},\Lambda^{\mathrm{hor}}](\alpha,\beta,\gamma)
=
\gamma([X^{\mathrm{hor}},Y^{\mathrm{hor}}]),
\]
where
\[
X^{\mathrm{hor}}=(\Lambda^{\mathrm{hor}})^\sharp(\alpha),
\quad
Y^{\mathrm{hor}}=(\Lambda^{\mathrm{hor}})^\sharp(\beta).
\]
Decomposing the bracket,
\[
[X^{\mathrm{hor}},Y^{\mathrm{hor}}]
=
[X^{\mathrm{hor}},Y^{\mathrm{hor}}]^{\mathrm{hor}}
+
\mathrm{Curv}_\Gamma(X,Y),
\]
and since $\gamma$ vanishes on $H$, we obtain
\[
\frac{1}{2}[\Lambda^{\mathrm{hor}},\Lambda^{\mathrm{hor}}](\alpha,\beta,\gamma)
=
\gamma(\mathrm{Curv}_\Gamma(X,Y)).
\]

\medskip

\noindent $\bullet$ {\bf Mixed term.}

Using the Koszul bracket,
\[
[\alpha,\beta]_\Lambda
=
\mathcal{L}_{\Lambda^\sharp(\alpha)}\beta
-
\mathcal{L}_{\Lambda^\sharp(\beta)}\alpha
-
d(\Lambda(\alpha,\beta)).
\]
Since $\alpha,\beta$ are horizontal,
\[
\Lambda^\sharp(\alpha)=X^{\mathrm{hor}},
\quad
\Lambda^\sharp(\beta)=Y^{\mathrm{hor}},
\quad
\Lambda(\alpha,\beta)=\Lambda^{\mathrm{hor}}(\alpha,\beta).
\]
Hence
\[
[\alpha,\beta]_\Lambda
=
\mathcal{L}_{X^{\mathrm{hor}}}\beta
-
\mathcal{L}_{Y^{\mathrm{hor}}}\alpha
-
d(\Lambda^{\mathrm{hor}}(\alpha,\beta)).
\]
Let $Z=(\Lambda^{\mathrm{ver}})^\sharp(\gamma)$. Since $\alpha,\beta$ are horizontal, their Lie derivatives remain horizontal and vanish on vertical vectors, so
\[
[\alpha,\beta]_\Lambda(Z)
=
- d(\Lambda^{\mathrm{hor}}(\alpha,\beta))(Z).
\]
Thus
\[
[\Lambda^{\mathrm{ver}},\Lambda^{\mathrm{hor}}](\alpha,\beta,\gamma)
=
-\gamma\!\left((\Lambda^{\mathrm{ver}})^\sharp(d(\Lambda^{\mathrm{hor}}(\alpha,\beta)))\right).
\]

\medskip

\noindent  Combining both contributions,
\[
\frac{1}{2}[\Lambda,\Lambda](\alpha,\beta,\gamma)
=
\gamma(\mathrm{Curv}_\Gamma(X,Y))
-
\gamma\!\left((\Lambda^{\mathrm{ver}})^\sharp(d(\Lambda^{\mathrm{hor}}(\alpha,\beta)))\right).
\]
On the other hand,
\[
(E\wedge\Lambda)(\alpha,\beta,\gamma)
=
\Lambda^{\mathrm{hor}}(\alpha,\beta)\,E(\gamma).
\]
Since $\gamma$ is arbitrary among vertical $1$-forms, the Jacobi identity is equivalent to
\[
\mathrm{Curv}_\Gamma(X,Y)
=
(\Lambda^{\mathrm{ver}})^\sharp\!\big(d(\Lambda^{\mathrm{hor}}(\alpha,\beta))\big)
+
\Lambda^{\mathrm{hor}}(\alpha,\beta)\,E.
\]
\end{proof}


\begin{proposition}\label{Proposition 3.1}
The pair $(\Lambda,E)$ defines a Jacobi structure on $M$
if and only if the following conditions hold:
\begin{enumerate}[label=(\alph*)]

\item
\[
[\Lambda^{\mathrm{ver}},\Lambda^{\mathrm{ver}}]
=
2\, E^{\mathrm{ver}} \wedge \Lambda^{\mathrm{ver}}.
\]

\item
\[
[\Lambda^{\mathrm{hor}}, E^{\mathrm{ver}}]=0.
\]

\item
The bivector field  $\Lambda^{\mathrm{hor}}$ is horizontally Poisson, i.e.
\[ [\Lambda^{\mathrm{hor}},\Lambda^{\mathrm{hor}}]
(\alpha,\beta,\gamma)=0
\]
for all horizontal $1$-forms $\alpha,\beta,\gamma$.

\item
For every $X \in \mathfrak{X}(B)$ and its horizontal lift $X^{\mathrm{hor}}$ satisfies:
\[
\mathcal{L}_{X^{\mathrm{hor}}}\Lambda^{\mathrm{ver}}=0.
\]

\item
The curvature of the connection $\Gamma$ satisfies
\[
\operatorname{Curv}_{\Gamma}(X,Y)
=
(\Lambda^{\mathrm{ver}})^\sharp\!\bigl(d\bigl(\Lambda^{\mathrm{hor}}(\alpha,\beta)\bigr)\bigr)
+
\Lambda^{\mathrm{hor}}(\alpha,\beta)\, E^{\mathrm{ver}},
\]
for all $X,Y \in \mathfrak{X}(B)$, where $\alpha,\beta$ are horizontal $1$-forms satisfying
\[
(\Lambda^{\mathrm{hor}})^\sharp(\alpha)=X^{\mathrm{hor}},
\qquad
(\Lambda^{\mathrm{hor}})^\sharp(\beta)=Y^{\mathrm{hor}}.
\]
The right-hand side is independent of the choice of such $\alpha,\beta$, since $\Lambda^{\mathrm{hor}}$ vanishes on vertical covectors.

\end{enumerate}
\end{proposition}
\begin{proof}
Recall that a pair $(\Lambda,E)$ defines a Jacobi structure on $M$ if and only if
\[
\frac{1}{2}[\Lambda,\Lambda]=E\wedge\Lambda,
\qquad
[\Lambda,E]=0,
\]
where $[\cdot,\cdot]$ denotes the Schouten bracket.
Let $\Gamma$ be the connection on $M$ induced by the principal connection on $P$.  We have:
\[
\Lambda=\Lambda^{\mathrm{ver}}+\Lambda^{\mathrm{hor}}, 
\qquad 
E= E^{\mathrm{ver}},
\]
where $\Lambda^{\mathrm{ver}}\in \Gamma(\wedge^2 V)$, $\Lambda^{\mathrm{hor}}\in \Gamma(\wedge^2 H)$, and $ E^{\mathrm{ver}}\in \Gamma(V)$.

Throughout the proof, we consider $1$-forms on $M$ which are \emph{projectable}, i.e. pullbacks of forms on the base or dual to horizontal and vertical distributions. In particular, horizontal $1$-forms vanish on the vertical distribution  $\mathcal V$, and their Lie derivatives along horizontal lifts remain horizontal.

\medskip

\noindent \textbf{(i) Necessary conditions.}
Assume that $(\Lambda,E)$ is a Jacobi structure.

\smallskip

\noindent \emph{(a) Case (V,V,V):}
Let $\alpha,\beta,\gamma$ be vertical $1$-forms. Since $\Lambda^{\mathrm{hor}}$ vanishes on vertical forms, and only $\Lambda^{\mathrm{ver}}$ contributes.
Hence
\[
[\Lambda^{\mathrm{ver}},\Lambda^{\mathrm{ver}}](\alpha,\beta,\gamma)
=
2\, E^{\mathrm{ver}}\wedge\Lambda^{\mathrm{ver}}(\alpha,\beta,\gamma).
\]

\smallskip

\noindent \emph{(b) Compatibility with $E$:}
From $[\Lambda,E]=0$, we obtain
\[
[\Lambda^{\mathrm{ver}}, E^{\mathrm{ver}}]
+
[\Lambda^{\mathrm{hor}}, E^{\mathrm{ver}}]
=0.
\]
Since $(\Lambda^{\mathrm{ver}}, E^{\mathrm{ver}})$ restricts on each fiber to a Jacobi structure, we have
\[
[\Lambda^{\mathrm{ver}}, E^{\mathrm{ver}}]=0,
\]
and thus
\[
[\Lambda^{\mathrm{hor}}, E^{\mathrm{ver}}]=0.
\]

\smallskip

\noindent \emph{(c) Case (H,H,H):}
Let $\alpha,\beta,\gamma$ be horizontal $1$-forms. Then $ E^{\mathrm{ver}}$ annihilates them, and $\Lambda^{\mathrm{ver}}$ vanishes on them. We have:
\[
(E\wedge\Lambda)(\alpha,\beta,\gamma)=0,
\]
and
\[
[\Lambda, \Lambda] (\alpha,\beta,\gamma)
=
[\Lambda^{\mathrm{hor}},\Lambda^{\mathrm{hor}}](\alpha,\beta,\gamma).
\]
Therefore
\[
[\Lambda^{\mathrm{hor}},\Lambda^{\mathrm{hor}}]=0
\]

\smallskip

\noindent \emph{(d) Case (H,V,V):}  
If $\alpha$ is horizontal and $\beta,\gamma$ are vertical, then
\[
(E\wedge\Lambda)(\alpha,\beta,\gamma)=0,
\]
since $E$ is vertical and $\Lambda$ has no mixed terms.  
Lemma~\ref{Lemme 3.2} then implies
\[
\left(\frac{1}{2}[\Lambda,\Lambda]-E\wedge\Lambda\right)(\alpha,\beta,\gamma)=0
\quad \Longleftrightarrow \quad
\mathcal{L}_{X^{\mathrm{hor}}}\Lambda^{\mathrm{ver}}=0.
\]

\smallskip

\noindent \emph{(e) Case (H,H,V):}  
If $\alpha,\beta$ are horizontal and $\gamma$ is vertical, Lemma~\ref{lem:curv-jacobi} shows that the Jacobi identity is equivalent to
\[
\operatorname{Curv}_\Gamma(X,Y)
=
(\Lambda^{\mathrm{ver}})^\sharp\!\bigl(d(\Lambda^{\mathrm{hor}}(\alpha,\beta))\bigr)
+
\Lambda^{\mathrm{hor}}(\alpha,\beta)\, E^{\mathrm{ver}}.
\]
Thus conditions (a)--(e) are necessary.

\medskip

\noindent\textbf{(ii) Sufficient conditions.}
Conversely, assume that conditions (a)--(e) hold. To check that  $[\Lambda,\Lambda]=2E\wedge\Lambda$, we evaluate both sides on all possible types of triples:
\begin{itemize}
\item \textbf{(V,V,V):} by (a),
\[
[\Lambda,\Lambda](\alpha,\beta,\gamma)
=
[\Lambda^{\mathrm{ver}},\Lambda^{\mathrm{ver}}](\alpha,\beta,\gamma)
=
2( E^{\mathrm{ver}}\wedge\Lambda^{\mathrm{ver}})(\alpha,\beta,\gamma).
\]
\item \textbf{(H,H,H):} by (c), 
\[
[\Lambda,\Lambda](\alpha,\beta,\gamma)=0,
\qquad
(E\wedge\Lambda)(\alpha,\beta,\gamma)=0.
\]
\item \textbf{(H,V,V):} both sides vanish, and by Lemma~\ref{Lemme 3.2}, condition (d) ensures equality.

\item \textbf{(H,H,V):} Lemma~\ref{lem:curv-jacobi} shows that the identity
\[
\frac{1}{2}[\Lambda,\Lambda]=E\wedge\Lambda
\]
is equivalent to condition (e).
\end{itemize}

\noindent These cases exhaust all possibilities, hence
\[
[\Lambda,\Lambda]=2E\wedge\Lambda.
\]

\medskip

\noindent Now we will check that $[\Lambda,E]=0$.
Decompose
\[
[\Lambda,E]
=
[\Lambda^{\mathrm{ver}}, E^{\mathrm{ver}}]
+
[\Lambda^{\mathrm{hor}}, E^{\mathrm{ver}}].
\]
By Relation (b),
\[
[\Lambda^{\mathrm{hor}}, E^{\mathrm{ver}}]=0.
\]
By (a), $(\Lambda^{\mathrm{ver}}, E^{\mathrm{ver}})$ defines a Jacobi structure on each fiber, hence
\[
[\Lambda^{\mathrm{ver}}, E^{\mathrm{ver}}]=0.
\]
Therefore,
\[
[\Lambda,E]=0.
\]

\medskip

\noindent Thus conditions (a)--(e) are necessary and sufficient for $(\Lambda,E)$ to define a Jacobi structure on $M$.
This completes the proof.

\end{proof}


\section{Example}
Now, we provide an example illustrating Theorem~\ref{MainTheorem}.
 Let $F=\mathbb{R}^3$ with coordinates $(x,y,z)$ and consider the tensors:
\[
\Lambda_F = x\,\partial_x \wedge \partial_y,
\qquad
E_F = \partial_y.
\]
A direct computation shows that
\[
[\Lambda_F,\Lambda_F]_{\mathrm{SN}} = 2 E_F \wedge \Lambda_F =0.
\]
Hence $(\Lambda_F,E_F)$ defines a Jacobi structure on $F$.
This structure is neither Poisson (since $E_F \neq 0$) nor contact, because $\Lambda_F$ is degenerate.

\medskip

\noindent $\bullet$
\emph{Group action.}
Consider the action of the group $G=(\mathbb{R},+)$ given by:
\[
t\cdot(x,y,z)=(x,y+t,z).
\]
The fundamental vector field associated with $X=1 \in \mathfrak g$ is
\[
X_F = \partial_y = E_F.
\]
Since $\Lambda_F$ and $E_F$ are invariant under translation in $y$, this action preserves the Jacobi structure:
\[
(\varphi_t)_*\Lambda_F=\Lambda_F,
\qquad
(\varphi_t)_*E_F=E_F.
\]

\medskip

\noindent $\bullet$ 
\emph{Hamiltonian structure.}
Define a moment map $\mu : F \to \mathfrak g^* \simeq \mathbb{R}$ by
\[
\mu \equiv 1.
\]
It is $G$-equivariant (the coadjoint action being trivial). Moreover,
\[
\Lambda_F^\sharp(d\mu)=0,
\qquad
\mu E_F = E_F,
\]
hence $X_F = \Lambda_F^\sharp(d\mu) + \mu E_F.$
Thus, the action is Hamiltonian in the sense of Definition~\ref{Hamiltonian Jacobi action}.

\medskip

\noindent $\bullet$
\emph{Principal bundle and connection.}
Let $(B,\omega_B)$ be a symplectic manifold of dimension $2m$. Consider the trivial principal bundle
\[
P = B \times \mathbb{R}
\]
with structure group $G=\mathbb{R}$. We choose a connection
$\omega \in \Omega^1(P,\mathfrak g)$ whose curvature satisfies:
\[
\mathrm{Curv}_\omega = \omega_B \otimes 1,
\]
which is always possible locally (and globally if $[\omega_B]$ is integral).

\medskip

\noindent $\bullet$ 
\emph{$\mu(F)$-fatness condition.}
Since $\mu \equiv 1$, for all $f \in F$ we have
\[
K_{\mu(f)}(X,Y)=\langle \mu(f),\mathrm{Curv}_\omega(X,Y)\rangle
=\omega_B(X,Y).
\]
Since $\omega_B$ is symplectic, the connection is $\mu(F)$-fat.

\medskip

\noindent $\bullet$
\emph{Structure on the associated bundle.}
The associated bundle
\[
M = P \times_G F \simeq B \times F
\]
carries a Jacobi structure $(\Lambda,E)$ provided by Theorem~\ref{MainTheorem}. Its vertical component coincides with $(\Lambda_F,E_F)$, while its horizontal component is determined by the curvature form via the coupling construction.

\medskip

\noindent
\emph{Explicit local expression.}
Let $(u^1,\dots,u^{2m})$ be local coordinates on $B$, and denote
\[
\Pi_B = \omega_B^{-1} = \sum_{i<j} \pi^{ij}(u)\,\partial_{u^i}\wedge\partial_{u^j}
\]
the Poisson bivector field  inverse to $\omega_B$. Then, in a local trivialization of $M$, the Jacobi structure reads
\[
\boxed{
\Lambda = x\,\partial_x \wedge \partial_y \;+\; \Pi_B^{\mathrm{hor}},
\qquad
E = \partial_y,
}
\]
where $\Pi_B^{\mathrm{hor}}$ denotes the horizontal lift of $\Pi_B$ defined by the connection $\omega$.

\medskip

\noindent
This example illustrates Theorem~\ref{MainTheorem} in a nontrivial case: the Jacobi structure on the fiber is neither Poisson nor contact, the curvature is nonzero, and the global structure arises from a coupling between the symplectic geometry of the base and the Jacobi geometry of the fiber.\\


\section{ Relation with existing constructions}

We show that the construction introduced here fits into a unified framework encompassing several classical theories, in particular the symplectic, locally conformal symplectic, and contact cases.

\subsection{The symplectic case}

Let $(F,\omega)$ be a symplectic manifold. The bivector field $\Lambda = \omega^{-1}$ and the vector field $E=0$ define a Jacobi structure (in fact, a Poisson structure).
A Hamiltonian action of a Lie group $G$ on $F$ is characterized by the existence of a moment map $\mu : F \to \mathfrak{g}^*$ such that, for every $X \in \mathfrak{g}$,
\[
i_{X_F}\omega = d\mu^X,
\qquad
X_F = \Lambda^\sharp(d\mu^X).
\]

\noindent In this setting, the nondegeneracy condition on the curvature (in the sense of Weinstein) coincides with that introduced in~\cite{Wei80}. Our construction yields a nondegenerate bivector  field on the total space $M$, while the field $E$ remains zero, so that the resulting structure is symplectic.
Thus we recover Weinstein's theorem~\cite[Thm.~3.2]{Wei80} as a special case.

\subsection{The locally conformal symplectic case}

\noindent Let $(F,\omega,\theta)$ be a locally conformal symplectic (LCS) manifold. It is well known that the pair
\[
\Lambda := \omega^{-1},
\qquad
E := \theta^\sharp
\]
defines a Jacobi structure satisfying
\[
[\Lambda,\Lambda] = 2E \wedge \Lambda,
\qquad
[\Lambda,E] = 0.
\]
An action is called twisted Hamiltonian \cite{HRS15} if
\[
i_{X_F}\omega = d\mu^X - \theta \wedge \mu^X.
\]
It follows that the fundamental vector field can be written as
\[
X_F = \Lambda^\sharp(d\mu^X) - \mu^X E,
\]
which corresponds to the natural formulation in Jacobi geometry.
In this framework, our construction applies directly and provides on $M$ a Jacobi structure of locally conformal symplectic type, whose Lee form is given by the pullback $\pi^*\theta$.
Thus we recover, as a special case, the main theorem of~\cite{HRS15}.

\subsection{The contact case}

\noindent Let $(F,\alpha)$ be a contact manifold of dimension $2n+1$. It is well known (see for instance \cite{Vai85}) that such a structure canonically determines a Jacobi structure $(\Lambda_F,E_F)$.
The Reeb vector field $R$ is defined by
\[
\alpha(R)=1,
\qquad
i_R d\alpha = 0,
\]
and we set $E_F = R$. The bivector field $\Lambda_F$ is characterized by the conditions $\Lambda_F^\sharp(\alpha)=0$ and, on the subbundle $\ker\alpha$, the map $\Lambda_F^\sharp$ is the inverse of the symplectic form $d\alpha|_{\ker\alpha}$. One then verifies that
\[
[\Lambda_F,\Lambda_F]_{\mathrm{SN}} = 2E_F \wedge \Lambda_F,
\qquad
[\Lambda_F,E_F]_{\mathrm{SN}} = 0,
\]
so that $(\Lambda_F,E_F)$ is a Jacobi structure.

\medskip

An action of $G$ by contact automorphisms is said to be Hamiltonian if there exists a $G$-equivariant moment map $\mu : F \to \mathfrak{g}^*$ such that
\[
i_{X_F}\alpha = \mu^X,
\qquad
X_F = \Lambda_F^\sharp(d\mu^X) + \mu^X R,
\]
where $X_F$ denotes the fundamental vector field associated with $X$.
This formulation is equivalent to the general notion of Hamiltonian action defined above.

\medskip

In this context, our theorem may be interpreted as an extension, to the Jacobi setting, of Lerman's constructions~\cite{Ler04} for contact bundles. In particular, the nondegeneracy condition on the curvature appearing in Proposition~\ref{Proposition 3.1} provides a Jacobi analogue of the condition introduced in these works, expressing compatibility between the curvature of the connection and the vertical Jacobi structure.

\medskip

Thus, our construction allows one to interpret the global contact structures obtained by Lerman within the more general framework of Jacobi geometry.

\subsection{The Poisson case}

In~\cite{AGAG08},  the author constructs a coupling Dirac structure on an associated bundle $P \times_G F$, starting from a connection and a Hamiltonian Poisson manifold $(F,\pi_F)$. Under a nondegeneracy condition on the curvature, this structure arises from a Poisson bivector  field on the total space.
Our theorem provides a counterpart of this construction in the Jacobi setting: a Hamiltonian action in Definition \ref{Hamiltonian Jacobi action}, combined with a suitable nondegeneracy condition, induces on $M$ a Jacobi structure compatible with those on the fibers. The presence of the Reeb vector field leads to additional compatibility conditions specific to this framework.


\section{Conclusion and perspectives:}

In this work, we extend the construction of coupling forms from symplectic, locally conformal symplectic, and contact geometry to the framework of Jacobi geometry. Our main result shows that a Hamiltonian action of a Lie group \(G\) on a Jacobi manifold \((F,\Lambda_F,E_F)\), together with a \(\mu(F)\)-fat connection on a principal \(G\)-bundle, induces a natural Jacobi structure on the associated bundle \(P \times_G F\). This construction provides a unified framework that encompasses and extends previous results due to Weinstein~\cite{Wei80}, Lerman~\cite{Ler04}, Otiman~\cite{HRS15}, and Wade~\cite{Wade2013}.

Since the formulation of Jacobi structures in terms of line bundles is more general than the classical definition due to Lichnerowicz, it is natural to investigate whether the present construction admits an extension to this broader setting. We expect that such an extension would further clarify the conceptual structure of the theory and broaden its range of applicability. This will be addressed in a forthcoming work.


\end{document}